\documentclass[11pt]{article}

\usepackage[margin=25mm]{geometry}
\usepackage{amsmath,amssymb,amsthm,mathtools,amsbsy}
\usepackage{multirow}
\usepackage{enumerate}
\usepackage{xurl}
\usepackage{xcolor}

\usepackage{hyperref}

\usepackage{cleveref}

\usepackage{amsbsy}
\def\ba{\pmb{a}}
\def\bb{\pmb{b}}
\def\bc{\pmb{c}}
\def\bd{\pmb{d}}
\def\be{\pmb{e}}

\def\bx{\pmb{x}}
\def\by{\pmb{y}}

\def\bone{\pmb{1}}

\def\cR{{\cal R}}

\newtheorem{theorem}{Theorem}[section]
\newtheorem{lemma}{Lemma}[section]
\newtheorem{corollary}{Corollary}[section]
\newtheorem{conjecture}{Conjecture}[section]

\theoremstyle{definition}
\newtheorem{definition}{Definition}[section]

\newcommand{\bbC}{\mathbb{C}}
\newcommand{\bbI}{\mathbb{I}}
\newcommand{\bbJ}{\mathbb{J}}
\newcommand{\bbN}{\mathbb{N}}
\newcommand{\bbP}{\mathbb{P}}
\newcommand{\bbR}{\mathbb{R}}

\DeclareMathOperator{\diag}{diag}

\DeclareMathOperator{\rank}{rank}
\DeclareMathOperator{\tr}{tr}
\DeclareMathOperator{\rmd}{d}
\DeclareMathOperator{\UI}{ui}

\DeclareMathOperator{\F}{F}
\DeclareMathOperator{\HH}{H}
\DeclareMathOperator{\T}{T}

\numberwithin{equation}{section}

\title{On a Conjecture Related to Eigenvalue Perturbations}

\author{
Lei-Hong Zhang%
\thanks{School of Mathematical Sciences, Soochow University, Suzhou 215006, Jiangsu, China. Email: {\tt longzlh@suda.edu.cn}.
             Supported in part by the National Natural Science Foundation of China
             NSFC-12471356, NSFC-12371380.}
\and
Ren-Cang Li%
\thanks{Department of Mathematics,
University of Texas at Arlington, Arlington, TX 76019-0408, USA.
Email: {\tt rcli@uta.edu.}}
}

\date{ }

\begin{document}
\maketitle

\begin{abstract}
In connection with eigenvalue perturbation bounds, Li
[\emph{Linear Algebra Appl.}, 278 (1998), pp.~317--326]
formulated three minimization problems involving Hadamard products.
The third asks how to relate
$$
\min_{W\ \mathrm{nonsingular}}\|W\circ G\|_\mathrm{ui}\,\|W^{-1}\circ G^\mathrm{T}\|_\mathrm{ui}
\quad\mbox{and}\quad
\min_{P\ \mathrm{permutation}}\|P\circ G\|_\mathrm{ui}^2,
$$
where $\|\cdot\|_\mathrm{ui}$ denotes a unitarily invariant norm and $\circ$
denotes the Hadamard product, i.e., entrywise multiplication.
The first quantity never exceeds the second, and a reverse inequality
with a constant independent of the matrix size was conjectured.
Romeo and Tilli [\emph{Linear Algebra Appl.}, 326 (2001), pp.~161--172]
showed that if such a bound holds for every $G$, then $\|\cdot\|_\mathrm{ui}$
must be uniformly equivalent to the Frobenius norm $\|\cdot\|_\mathrm{F}$.
We prove that the two quantities coincide when $\|\cdot\|_\mathrm{ui}=\|\cdot\|_\mathrm{F}$:
for every complex matrix $G$,
$$
\min_{W\ \mathrm{nonsingular}}
\|W\circ G\|_\mathrm{F}\,\|W^{-1}\circ G^\mathrm{T}\|_\mathrm{F}
=\min_{P\ \mathrm{permutation}}\|P\circ G\|_\mathrm{F}^2.
$$
The associated conjecture therefore holds with the optimal constant $c=1$.
%In connection with eigenvalue perturbation bounds, Li
%[\emph{Linear Algebra Appl.}, 278 (1998), pp.~317--326]
%formulated three minimization problems involving Hadamard products.
%The third asks how to relate
%$$
%\min_{W\ \mathrm{nonsingular}}\|W\circ G\|_{\UI}\,\|W^{-1}\circ G^{\T}\|_{\UI}
%\quad\mbox{and}\quad
%\min_{P\ \mathrm{permutation}}\|P\circ G\|_{\UI}^2,
%$$
%where $\|\cdot\|_{\UI}$ denotes a unitarily invariant norm and $\circ$
%denotes the Hadamard product, i.e., entrywise multiplication.
%The first quantity never exceeds the second, and a reverse inequality
%with a constant independent of the matrix size was conjectured.
%Romeo and Tilli [\emph{Linear Algebra Appl.}, 326 (2001), pp.~161--172]
%showed that if such a bound holds for every $G$, then $\|\cdot\|_{\UI}$
%must be uniformly equivalent to the Frobenius norm $\|\cdot\|_{\F}$.
%We prove that the two quantities coincide when $\|\cdot\|_{\UI}=\|\cdot\|_{\F}$:
%for every complex matrix $G$,
%$$
%\min_{W\ \mathrm{nonsingular}}
%\|W\circ G\|_{\F}\,\|W^{-1}\circ G^{\T}\|_{\F}
%=\min_{P\ \mathrm{permutation}}\|P\circ G\|_{\F}^2.
%$$
%The associated conjecture therefore holds with the optimal constant $c=1$.
\end{abstract}



%%%%%%%%%%%%%
\section{Introduction}
\label{sec:intro}

The spectral variation of a matrix has been a very active research subject
in both matrix theory and numerical linear algebra.
In~\cite{li:1998} it was observed that a number of perturbation results
for eigenvalues of matrices/matrix pencils can be recast as
minimization problems involving Hadamard products (see also \cite{li:1993a}).
These reformulations were not the main goal of that paper.
What was hoped for was a possibly unifying theory behind these perturbation results.

Let $G=[g_{ij}]\in\bbC^{n\times n}$ be given, and
let $X\circ Y=[x_{ij}y_{ij}]$ denote the matrix Hadamard product of $X=[x_{ij}]$
and $Y=[y_{ij}]$ of the same size.
Denote by $\|\cdot\|_{\UI}$  any unitarily invariant norm, normalized so that
it agrees with $\|\cdot\|_2$ on rank-one matrices. Two special unitarily invariant norms
are the matrix spectral norm $\|\cdot\|_2$ and the matrix Frobenius norm $\|\cdot\|_{\F}$.

The following three problems were put forward in~\cite{li:1998}.

\medskip
\noindent{\bf Problem 1.}
Relate $\displaystyle\min_{Q\ \mathrm{unitary}}\|Q\circ G\|_{\UI}$
to $\displaystyle\min_{P\ \mathrm{permutation}}\|P\circ G\|_{\UI}$.

\medskip
\noindent{\bf Problem 2.}
Relate $\displaystyle\min_{W\ \mathrm{nonsingular}}\|W^{-1}\|_2\,\|W\circ G\|_{\UI}$
to $\displaystyle\min_{P\ \mathrm{permutation}}\|P\circ G\|_{\UI}$.

\medskip
\noindent{\bf Problem 3.}
Relate $\displaystyle\min_{W\ \mathrm{nonsingular}}\|W\circ G\|_{\UI}\,\|W^{-1}\circ G^{\T}\|_{\UI}$
to $\displaystyle\min_{P\ \mathrm{permutation}}\|P\circ G\|_{\UI}^2$.

\medskip

A trivial relation in each problem is that the first expression is no bigger
than the second, because every permutation matrix is unitary, hence
nonsingular.
In general, Li's conjectures \cite{li:1998} ask for reverse inequalities with
constants independent of $n$ and of $G$. Specifically, for
Problems~2 and 3, the conjectures are stated below.

\begin{conjecture}\label{conj:li1998}
There are constants $c_1\ge 1$ and $c_2\ge 1$, both possibly dependent on
the unitarily invariant norm $\|\cdot\|_{\UI}$, but independent of the matrix sizes and $G$, such that
\begin{align}
c_1\cdot\min_{W\ \mathrm{nonsingular}}\|W^{-1}\|_2\,\|W\circ G\|_{\UI}
   &\ge\min_{P\ \mathrm{permutation}}\|P\circ G\|_{\UI}, \label{eq:liconj-1}  \\
c_2\cdot\min_{W\ \mathrm{nonsingular}}\|W\circ G\|_{\UI}\,\|W^{-1}\circ G^{\T}\|_{\UI}
   &\ge\min_{P\ \mathrm{permutation}}\|P\circ G\|_{\UI}^2. \label{eq:liconj-2}
\end{align}
\end{conjecture}

Although not explicitly formulated, a natural conjecture about Problem~1 is that there exists a constant $c_0$,
independent of the matrix sizes, such that
\begin{equation}\label{eq:liconj-0}
c_0\cdot\min_{W\ \mathrm{unitary}}\|W\circ G\|_{\UI}
   \ge\min_{P\ \mathrm{permutation}}\|P\circ G\|_{\UI}.
\end{equation}
For the Frobenius norm, the proof technique of Hoffman and Wielandt \cite{howi:1953} immediately yields
\begin{align}
\min_{W\ \mathrm{unitary}}\|W\circ G\|_{\F}
   &=\min_{P\ \mathrm{permutation}}\|P\circ G\|_{\F}, \label{eq:p1F:unitary} \\
\min_{W\ \mathrm{unitary}}\|W\circ G\|_{\F}\,\|W^{-1}\circ G^{\T}\|_{\F}
             &=\min_{P\ \mathrm{permutation}}\|P\circ G\|_{\F}^2. \label{eq:p2F:unitary}
\end{align}
Also for the Frobenius norm,
Problem~2 was solved in~\cite{li:1998}:
\begin{equation}\label{eq:p2F}
\min_{W\ \mathrm{nonsingular}}\|W^{-1}\|_2\,\|W\circ G\|_{\F}
=\min_{P\ \mathrm{permutation}}\|P\circ G\|_{\F},
\end{equation}
with the help of a theorem of Elsner and Friedland~\cite{elfr:1995}
and again the proof technique in \cite{howi:1953}.

Yet, the Frobenius norm case of Problem~3, i.e., conjecture \eqref{eq:liconj-2}, remains open.
The purpose of this paper is to settle \eqref{eq:liconj-2}
for the Frobenius norm.
Our main result is
\begin{theorem}
\label{thm:main}
For every $n\ge 1$ and every $G\in\bbC^{n\times n}$,
\begin{equation}
\label{eq:main}
\min_{W\ \mathrm{nonsingular}}
  \|W\circ G\|_{\F}\,\|W^{-1}\circ G^{\T}\|_{\F}
  =\min_{P\ \mathrm{permutation}}\|P\circ G\|_{\F}^2.
\end{equation}
\end{theorem}

Throughout this paper, we will consider unitarily invariant norms $\|\cdot\|_{\UI}$ that are generic to matrix sizes, i.e.,
the norms can be applied to matrix of any sizes. Examples include the matrix spectral norm $\|\cdot\|_2$,
the matrix Frobenius norm $\|\cdot\|_{\F}$,
and more generally, the matrix Schatten $p$-norm $\|\cdot\|_{S_p}$ ($1\le p\le \infty$).

 The rest of this paper is organized as follows.
In \cref{sec:roti2001}, we discuss the  substantial progress by Tilli~\cite{till:2001} and Romeo and Tilli~\cite{roti:2001}
towards \Cref{conj:li1998} and their immediate implications.
To prepare our proof of Theorem~\ref{thm:main} in \cref{sec:proof},  \cref{sec:notation} collects
some preliminaries, especially on doubly superstochastic matrices,
and proves
that a quadratic
mean of the entrywise squares of $W$ and of $W^{-\T}$ is doubly
superstochastic.
In \cref{sec:discussion}, we will discuss some special cases of our conjectures for those $G$ coming from
eigenvalue perturbation theory and directions for further research. Finally,
conclusions are drawn in \cref{sec:concl}.

{\bf Notation.}
Throughout, $\bbC^{m\times n}$ denotes the set of $m\times n$ complex
matrices, while $\bbR^{m\times n}$ is the set of $m\times n$ real
matrices.
The $n\times n$ identity matrix is $I_n$, or simply $I$ when its size is clear from the context, and $\be_1,\ldots,\be_n$ form the standard basis
of $\bbC^n$, or equivalently, the columns   of $I_n$.
$\cR(X)$ stands for the column space of $X$.
The complex conjugate transpose is written as $X^{\HH}$, while $(\cdot)^{\T}$ takes the
matrix/vector transpose.

\section{Romeo-Tilli Theorem and Implications}\label{sec:roti2001}
Limiting to unitary matrices,
Tilli~\cite{till:2001} and Romeo and Tilli~\cite{roti:2001} made substantial progress towards
\Cref{conj:li1998}. Specifically, they proved

\begin{theorem}[{Romeo and Tilli~\cite{roti:2001}}]
\label{thm:roti2001}
Let $\|\cdot\|_{\UI}$ be a unitarily invariant norm.
\begin{enumerate}[{\rm (a)}]
  \item If there exists a constant $c_1\ge 1$, independent of matrix size, such that
        \begin{equation}\label{eq:liconj-1'}
        c_1\cdot\min_{W\ \mathrm{unitary}}\|W^{-1}\|_2\,\|W\circ G\|_{\UI}
            \ge\min_{P\ \mathrm{permutation}}\|P\circ G\|_{\UI},
        \end{equation}
        then the unitarily invariant norm must be uniformly equivalent to
        the Frobenius norm:
        \begin{equation}\label{eq:liconj-1'cond}
        c_1^{-1}\|\cdot\|_{\F}\le\|\cdot\|_{\UI}\le c_1\|\cdot\|_{\F}.
        \end{equation}
  \item  If there exists a constant $c_2\ge 1$, independent of matrix size, such that
        \begin{equation}\label{eq:liconj-2'}
        c_2\cdot\min_{W\ \mathrm{unitary}}\|W\circ G\|_{\UI}\,\|W^{-1}\circ G^{\T}\|_{\UI}
             \ge\min_{P\ \mathrm{permutation}}\|P\circ G\|_{\UI}^2
        \end{equation}
        then the unitarily invariant norm must be uniformly equivalent to
        the Frobenius norm:
        \begin{equation}\label{eq:liconj-2'cond}
        c_2^{-1/2}\|\cdot\|_{\F}\le\|\cdot\|_{\UI}\le c_2^{1/2}\|\cdot\|_{\F}.
        \end{equation}
\end{enumerate}
\end{theorem}

In \Cref{thm:roti2001}, the minimization is restricted to  unitary matrices. Since the set of unitary matrices is
a subset of nonsingular matrices, an immediate corollary is that \Cref{thm:roti2001} remains valid if \eqref{eq:liconj-1'} and \eqref{eq:liconj-2'} are replaced with \eqref{eq:liconj-1} and \eqref{eq:liconj-2}, respectively.

\begin{corollary}\label{cor:roti2001}
If \eqref{eq:liconj-1} holds for any $G\in\bbC^{n\times n}$ and for some constant $c_1\ge 1$, independent of $n$, then
\eqref{eq:liconj-1'cond} holds. If \eqref{eq:liconj-2} holds for any $G\in\bbC^{n\times n}$ and for some constant $c_2\ge 1$, independent of $n$, then
\eqref{eq:liconj-2'cond} holds.
\end{corollary}

\begin{proof}
If \eqref{eq:liconj-1} holds for any $G\in\bbC^{n\times n}$ and for some constant $c_1\ge 1$, independent of $n$, then
$$
c_1\cdot\min_{W\ \mathrm{unitary}}\|W^{-1}\|_2\,\|W\circ G\|_{\UI}
            \ge c_1\cdot\min_{W\ \mathrm{nonsingular}}\|W^{-1}\|_2\,\|W\circ G\|_{\UI}
            \ge\min_{P\ \mathrm{permutation}}\|P\circ G\|_{\UI},
$$
yielding \eqref{eq:liconj-1'} and therefore \eqref{eq:liconj-1'cond} by
\Cref{thm:roti2001}. Similarly, we can get the second claim in this corollary.
\end{proof}
 
Combining \Cref{thm:roti2001}(a) with either \eqref{eq:p1F:unitary} or \eqref{eq:p2F}, we conclude that
\begin{corollary}\label{cor:roti2001'}
Conjecture \eqref{eq:liconj-1'} holds if and only if \eqref{eq:liconj-1'cond} holds;
conjecture~\eqref{eq:liconj-1} holds if and only if \eqref{eq:liconj-1'cond}   holds.
\end{corollary}

\begin{proof}
That conjecture \eqref{eq:liconj-1'} implies \eqref{eq:liconj-1'cond} is the conclusion
of \Cref{thm:roti2001}(a), and that \eqref{eq:liconj-1} implies \eqref{eq:liconj-1'cond} follows from
\Cref{cor:roti2001}. On the other hand, if \eqref{eq:liconj-1'cond} holds, then
$$
c_1\|W^{-1}\|_2\,\|W\circ G\|_{\UI}\ge\|W^{-1}\|_2\,\|W\circ G\|_{\F},
$$
minimizing both sides of which over unitary (resp., nonsingular) $W$, with the help of
\eqref{eq:p1F:unitary} (resp., \eqref{eq:p2F}), we arrive at \eqref{eq:liconj-1'}
(resp., \eqref{eq:liconj-1}).
\end{proof}

%%%%%%%%%%%
\section{Preliminaries}
\label{sec:notation}
The singular values of a matrix $X\in\bbC^{n\times n}$ are denoted by
$\sigma_1(X)\ge\cdots\ge\sigma_n(X)$, arranged in nonincreasing order.
For $1\le p<\infty$, the Schatten $p$-norm is
$$
\|X\|_{S_p}=\Biggl(\sum_{i=1}^n\bigl[\sigma_i(X)\bigr]^p\Biggr)^{1/p},
$$
while $\|X\|_{S_\infty}=\lim_{p\to\infty}\|X\|_{S_p}=\sigma_1(X)=\|X\|_2$.
In particular $\|X\|_{\F}=\|X\|_{S_2}$, and the trace norm, also known as the nuclear norm, is
$\|X\|_{S_1}=\sum_{i=1}^n\sigma_i(X)$.
Write $\lambda_1(X),\ldots,\lambda_n(X)$ for the eigenvalues of $X$,
counted with algebraic multiplicity and taken in any order.

Weyl's majorization of  the absolute values of the eigenvalues by the singular values
\cite[p.42]{bhat:1997}
asserts that, for any $X\in\bbC^{n\times n}$,
\begin{equation}
\label{eq:weyl}
\sum_{i=1}^n\bigl|\lambda_i(X)\bigr|
\le\sum_{i=1}^n\sigma_i(X)=\|X\|_{S_1}.
\end{equation} 
We also have, the so-called H\"older's inequality for Schatten norms,
\begin{equation}
\label{eq:holder}
\|XY\|_{S_1}\le\|X\|_{S_p}\,\|Y\|_{S_q}
\end{equation}
for any $1\le p,\, q\le\infty$ such that $1/p+1/q=1$.
A quick proof of \eqref{eq:holder} is to use von Neumann's trace inequality \cite{neum:1937} (see also \cite[p.183]{hojo:1991}) and
H\"older's inequality. Let $U$ be some unitary matrix ix such that $XYU$ is Hermitian and positive semidefinite and has the singular values of $XY$ as its eigenvalues. We have
\begin{align*}
\|XY\|_{S_1} =\tr(XYU)
   &\le\sum_i\sigma_i(X)\sigma_i(YU) \qquad\mbox{(use von Neumann's trace inequality)}\\
   &=\sum_i\sigma_i(X)\sigma_i(Y) \qquad\mbox{(use $\sigma_i(YU)=\sigma_i(Y)$)} \\
   &\le\Big(\sum_i[\sigma_i(X)]^p\Big)^{1/p}\Big(\sum_i[\sigma_i(Y)]^q\Big)^{1/q},
                  \qquad\mbox{(use H\"older's inequality)}
\end{align*}
as was to be shown, where $\tr(\cdot)$ takes the trace of a square matrix. Letting $p=q=2$ in \eqref{eq:holder} leads to
\begin{equation}
\label{eq:holder-F}
\|XY\|_{S_1}\le\|X\|_{\F}\,\|Y\|_{\F}
\end{equation}
which we will use below.

Let $\bbN_n=\{1,2,\ldots,n\}$.
For $\bbI\subset\bbN_n$ we write $|\bbI|$ for the
cardinality of $\bbI$
and let $E_{\bbI}\in\bbC^{n\times n}$ be the diagonal matrix whose $i$th diagonal
entry is $1$
if $i\in\bbI$ and $0$ otherwise, i.e.,
$E_{\bbI}$ is the orthogonal projection onto the coordinate
subspace $\cR(E_{\bbI})=\mathrm{span}\{\be_i:i\in\bbI\}$.
We write $\bbP_n$ for the set of $n\times n$ permutation matrices.
If $P\in\bbP_n$ corresponds to a permutation $\tau$ of
$\bbN_n$, then
\begin{equation}
\label{eq:assign}
\|P\circ G\|_{\F}^2=\sum_{i=1}^n\bigl|g_{i,\tau(i)}\bigr|^2,
\end{equation}
where $G\equiv [g_{ij}]$.
Thus the right-hand side of~\eqref{eq:main} is the minimal assignment cost
associated with the nonnegative matrix $C=[|g_{ij}|^2]$.

For $x,y\ge 0$, we have
\begin{equation}\label{eq:amgm}
\sqrt{xy}=\inf_{t>0}\frac12\bigl(tx+t^{-1}y\bigr).
\end{equation}
Indeed, for every $t>0$,
\begin{equation}\label{eq:amgm'}
\frac12\bigl(tx+t^{-1}y\bigr)-\sqrt{xy}
=\frac12\Bigl(\sqrt{tx}-\sqrt{t^{-1}y}\Bigr)^2\ge 0,
\end{equation}
so $\sqrt{xy}$ does not exceed the infimum on the right of~\eqref{eq:amgm}.
If $x>0$ and $y>0$, the squared term in \eqref{eq:amgm'} vanishes at $t=\sqrt{y/x}$, and the
infimum is a minimum.
If $x=0$, then $\frac12 t^{-1}y\to 0$ as $t\to\infty$;
if $y=0$, then $\frac12 tx\to 0$ as $t\to 0^+$.
In either case the infimum is $0=\sqrt{xy}$.

%\section{Doubly superstochastic matrices}
%\label{sec:dss}

Next we introduce the notion of a {\em doubly superstochastic matrix}
(see \cite[Chapter~2, Section~D]{marshall:2011} and also
\cite{ando:1989}).

\begin{definition}
\label{def:dss}
A nonnegative matrix $H\in\bbR^{n\times n}$ is called {\em doubly superstochastic}
if there exists a doubly stochastic matrix $S$ such that $0\le S\le H$
entrywise.
\end{definition}

A matrix is doubly stochastic when it is nonnegative and every row sum and
every column sum equals one. According to this definition,
row and column sums of a doubly superstochastic matrix are necessarily at
least one, but that necessary condition is not
sufficient~\cite{marshall:2011}.
A complete criterion is given in the next lemma.
%due to Gale~\cite{gale:1957}.
%Necessity follows from the definition by a direct estimate; the converse is
%Gale's feasibility theorem~\cite[p.1075]{gale:1957}.

\begin{lemma}[{\cite[Corollary~3.4]{ando:1989}}]
\label{lem:cut}
A nonnegative matrix $H=[h_{ij}]\in\bbR^{n\times n}$ is doubly superstochastic if
and only if
\begin{equation}
\label{eq:cut}
\sum_{i\in\bbI,\,j\in\bbJ}h_{ij}\ge |\bbI|+|\bbJ|-n
\end{equation}
for all subsets $\bbI,\bbJ\subset\bbN_n$.
\end{lemma}
 
\begin{lemma}
\label{lem:cost}
Let $H$ be doubly superstochastic and let $C=[c_{ij}]$ be nonnegative.
Then
\begin{equation}
\label{eq:cost}
\sum_{i,j}h_{ij}c_{ij}
\ge\min_{\tau}\sum_{i=1}^n c_{i,\tau(i)},
\end{equation}
the minimum being taken over all permutations $\tau$ of $\bbN_n$.
\end{lemma}

\begin{proof}
There is a doubly stochastic matrix $S$ such that $0\le S\le H$.
Since $C\ge 0$,
$$
\sum_{i,j}h_{ij}c_{ij}\ge\sum_{i,j}s_{ij}c_{ij}.
$$
By the Birkhoff--von Neumann theorem~\cite{birkhoff:1946,vonneumann:1953},
$S$ is a convex combination of permutation matrices.
The right-hand side is therefore a convex combination of assignment costs,
and~\eqref{eq:cost} follows.
\end{proof}

The next lemma is the analytic heart of the argument.
For real nonsingular $W$ it is a special case of a
theorem\footnote {Apply \cite[Theorem~6]{bhja:2015} with $A=t^{1/2}W$, $G=t^{-1/2}W^{-\T}$, and $B=C=0$ there.}
of
Bhatia and
Jain on symplectic matrices \cite[Theorem~6]{bhja:2015}.
% (see section~\ref{sec:remarks}).
The proof given here does not use the symplectic group, and it applies to
every complex nonsingular $W$.

\begin{lemma}
\label{lem:Ht}
Let $W\equiv[w_{ij}]\in\bbC^{n\times n}$ be nonsingular and let $t>0$.
Then the matrix
\begin{equation}
\label{eq:Ht}
H(t)=\frac12\bigl(t\,|W|^{\circ 2}+t^{-1}|W^{-\T}|^{\circ 2}\bigr)
   \equiv\frac 12\bigl[t\,|w_{ij}|^2+t^{-1}\bigl|(W^{-1})_{ji}\bigr|^2\bigr]
\end{equation}
is doubly superstochastic,
where $(W^{-1})_{ji}$ denotes the $(j,i)$th entry of $W^{-1}$, $|W|$ takes entrywise absolute value, and
$|W|^{\circ 2}=|W|\circ |W|$.
\end{lemma}

\begin{proof}
By Lemma~\ref{lem:cut}, it suffices to verify~\eqref{eq:cut} for every pair
of index sets $\bbI,\bbJ\subset\bbN_n$ with $|\bbI|+|\bbJ|\ge n$.
Assume therefore that
$$
k:=|\bbI|+|\bbJ|-n>0.
$$
Set
$A=E_{\bbI}W E_{\bbJ}$ and $B=E_{\bbJ}W^{-1}E_{\bbI}$.
Then $A,\,B\in\bbC^{n\times n}$, and a direct inspection of entries gives
$$
\|A\|_{\F}^2
=\sum_{i\in\bbI,\,j\in\bbJ}|w_{ij}|^2,\quad
\|B\|_{\F}^2
=\sum_{i\in\bbI,\,j\in\bbJ}\bigl|(W^{-1})_{ji}\bigr|^2.
$$
Write $H(t)\equiv [h_{ij}(t)]$ as in \eqref{eq:Ht}. We have
\begin{equation}
\label{eq:block}
\sum_{i\in\bbI,\,j\in\bbJ}h_{ij}(t)
=\frac12\Bigl(t\,\|A\|_{\F}^2+t^{-1}\|B\|_{\F}^2\Bigr)
\ge\|A\|_{\F}\,\|B\|_{\F},
\end{equation}
where the last step follows from~\eqref{eq:amgm} with $x=\|A\|_{\F}^2$ and
$y=\|B\|_{\F}^2$.

Consider the subspace
$$
\mathcal{L}:=\cR(E_{\bbI})\cap W(\cR(E_{\bbJ}))\subset\bbC^n.
$$
Since $W$ is nonsingular, $\dim W(\cR(E_{\bbJ}))=|\bbJ|$.
The dimension formula for a pair of subspaces of $\bbC^n$ therefore yields
\begin{align*}
n\ge \dim(\cR(E_{\bbI})+ W(\cR(E_{\bbJ})))&=\dim\cR(E_{\bbI})+\dim W(\cR(E_{\bbJ}))-\dim\mathcal{L} \\
   &=|\bbI|+|\bbJ|-\dim\mathcal{L},
\end{align*}
leading to
$$ %\begin{equation}\label{eq:dim}
\dim\mathcal{L}
\ge |\bbI|+|\bbJ|-n=k>0.
$$ %\end{equation}
In particular $\mathcal{L}\neq\{0\}$, so neither $A$ nor $B$ can vanish.

Now let $\bx\in\mathcal{L}$. Then $\bx\in\cR(E_{\bbI})$ and thus $E_{\bbI}\bx=\bx$, and
also $\bx\in W(\cR(E_{\bbJ}))$ and thus there exists
$\by\in\cR(E_{\bbJ})$ (implying $\by=E_{\bbJ}\by$) such that $\bx=W\by$ (so $\by=W^{-1}\bx$).
Thus $\by=E_{\bbJ}\by=W^{-1}\bx$, and therefore
$$
AB\bx
=E_{\bbI}W \underbrace{E_{\bbJ}^2}_{=E_{\bbJ}}W^{-1}\underbrace{E_{\bbI}\bx}_{=\bx}
=E_{\bbI}W E_{\bbJ}\underbrace{W^{-1}\bx}_{=\by}
=E_{\bbI}W \underbrace{E_{\bbJ}\by}_{=\by}
=E_{\bbI}\underbrace{W\by}_{=\bx}
=E_{\bbI}\bx
=\bx.
$$
So $AB$ acts as the identity on $\mathcal{L}$.
The geometric multiplicity of the eigenvalue $1$ of $AB$ is therefore at
least $\dim\mathcal{L}\ge k$.
The algebraic multiplicity is at least as large as the geometric
multiplicity, and we conclude that at least $k$ eigenvalues of $AB$,
counted algebraically, are equal to $1$.
Consequently
$$ %\begin{equation}\label{eq:ev}
\sum_{\ell=1}^n\bigl|\lambda_\ell(AB)\bigr|\ge k.
$$ %\end{equation}
Inequalities~\eqref{eq:weyl} and~\eqref{eq:holder-F} now give
\begin{equation}
\label{eq:nuclear}
k
\le\sum_{\ell=1}^n\bigl|\lambda_\ell(AB)\bigr|
\le\|AB\|_{S_1}
\le\|A\|_{\F}\,\|B\|_{\F}.
\end{equation}
Combining~\eqref{eq:block} with~\eqref{eq:nuclear}, we obtain
$$
\sum_{i\in\bbI,\,j\in\bbJ}h_{ij}(t)\ge k=|\bbI|+|\bbJ|-n.
$$
By Lemma \ref{lem:cut}, $H(t)$ is doubly superstochastic.
\end{proof}

Two comments on the matrices $A$ and $B$ in the proof above are in order.
First, they are regarded as operators on the same space $\bbC^n$, so that
the product $AB$, its eigenvalues, and $\|AB\|_{S_1}$ are unambiguously
defined.
Their nonzero blocks coincide with the ordinary submatrices of $W$ and of
$W^{-1}$ indexed by $\bbI\times\bbJ$ and $\bbJ\times\bbI$, respectively.
Second, $AB$ need not be normal.
The passage from eigenvalues to singular values in~\eqref{eq:nuclear} uses
only~\eqref{eq:weyl}.

%%%%%%%%%%%%%
\section{Proof of \Cref{thm:main}}
\label{sec:proof}

Let $C=[c_{ij}]$ where $c_{ij}=|g_{ij}|^2\ge 0$.
For an arbitrary nonsingular $W$, set
\begin{align}
x&:=\|W\circ G\|_{\F}^2
=\sum_{i,j}|w_{ij}|^2c_{ij}, \label{eq:xy-x}\\
y&:=\|W^{-1}\circ G^{\T}\|_{\F}^2
=\sum_{i,j}\bigl|(W^{-1})_{ij}\bigr|^2|g_{ji}|^2
=\sum_{i,j}\bigl|(W^{-1})_{ji}\bigr|^2c_{ij}. \label{eq:xy-y}
\end{align}
The last equality is a relabeling of the summation indices.
Equivalently, $y=\|W^{-\T}\circ G\|_{\F}^2$.

We first assume $x>0$ and $y>0$.
For each $t>0$, \Cref{lem:cost,lem:Ht} give
\begin{equation}
\label{eq:preopt}
\frac12\bigl(tx+t^{-1}y\bigr)
=\sum_{i,j}h_{ij}(t)\,c_{ij}
\ge\min_{\tau}\sum_{i=1}^n c_{i,\tau(i)}
=\min_{P\in\bbP_n}\|P\circ G\|_{\F}^2.
\end{equation}
Taking the infimum over $t>0$ and using~\eqref{eq:amgm}, we obtain
\begin{equation}
\label{eq:lower}
\|W\circ G\|_{\F}\,\|W^{-1}\circ G^{\T}\|_{\F}
=\sqrt{xy}
\ge\min_{P\in\bbP_n}\|P\circ G\|_{\F}^2.
\end{equation}

It remains to treat the cases in which $x=0$ or $y=0$.
Suppose $x=0$. Then $w_{ij}\neq 0$ implies $g_{ij}=0$.
Since $W$ is nonsingular, $\det W\neq 0$.
In the Leibniz expansion of $\det W$, at least one summand is therefore
nonzero, and there exists a permutation $\tau$ such that
$w_{i,\tau(i)}\neq 0$ for every $i$.
Hence $g_{i,\tau(i)}=0$ for every $i$, and thus
$\min_{P\in\bbP_n}\|P\circ G\|_{\F}^2=0$.
The inequality~\eqref{eq:lower} continues to hold.
If $y=0$, the same argument applied to the nonsingular matrix
$W^{-\T}$ yields the same conclusion, because
$y=\|W^{-\T}\circ G\|_{\F}^2$.

Thus~\eqref{eq:lower} is valid for every nonsingular $W$.
Taking the infimum over $W$ gives
\begin{equation}
\label{eq:oneside}
\inf_{W\ \mathrm{nonsingular}}
\|W\circ G\|_{\F}\,\|W^{-1}\circ G^{\T}\|_{\F}
\ge\min_{P\in\bbP_n}\|P\circ G\|_{\F}^2.
\end{equation}

On the other hand, it is trivial to see that the left-hand side
of~\eqref{eq:oneside} is no bigger than the right-hand side, because the
permutation matrices are a subset of the nonsingular ones.
More explicitly, let $P\in\bbP_n$ attain the minimum on the right of
\eqref{eq:main}.
Then $P^{-1}=P^{\T}$, and
\begin{align*}
\|P\circ G\|_{\F}\,\|P^{-1}\circ G^{\T}\|_{\F}
&=\|P\circ G\|_{\F}\,\|P^{\T}\circ G^{\T}\|_{\F}
=\|P\circ G\|_{\F}\,\|(P\circ G)^{\T}\|_{\F}\\
&=\|P\circ G\|_{\F}^2
=\min_{Q\in\bbP_n}\|Q\circ G\|_{\F}^2.
\end{align*}
Hence the infimum in~\eqref{eq:oneside} is a minimum, it is attained at $P$,
and~\eqref{eq:main} follows.

%%%%%%%%%%%%
\section{Discussion}\label{sec:discussion}
Similarly to \Cref{cor:roti2001'}, \Cref{thm:main} implies the converse of
\Cref{thm:roti2001}(b).
Combining \Cref{thm:roti2001}(b) with   \eqref{eq:main}, we conclude
\Cref{cor:main} below.
%\eqref{eq:liconj-2'} and \eqref{eq:liconj-2'cond} are equivalent and that
%\eqref{eq:liconj-2} and \eqref{eq:liconj-2'cond} are equivalent.

\begin{corollary}\label{cor:main}
Conjecture \eqref{eq:liconj-2'} holds if and only if \eqref{eq:liconj-2'cond} holds;
conjecture~\eqref{eq:liconj-2} holds if and only if \eqref{eq:liconj-2'cond} holds.
\end{corollary}

\begin{proof}
That conjecture \eqref{eq:liconj-2'} implies \eqref{eq:liconj-2'cond} is the conclusion
of \Cref{thm:roti2001}(b), and that \eqref{eq:liconj-2} implies \eqref{eq:liconj-2'cond} follows from
\Cref{cor:roti2001}. On the other hand, if \eqref{eq:liconj-2'cond} holds, then
$$
c_2\|W\circ G\|_{\UI}\,\|W^{-1}\circ G^{\T}\|_{\UI}\ge\|W\circ G\|_{\F}\,\|W^{-1}\circ G^{\T}\|_{\F},
$$
minimizing both sides of which over unitary (resp., nonsingular) $W$, with the help of
\eqref{eq:p2F:unitary} (resp., \eqref{eq:main}), we arrive at \eqref{eq:liconj-2'}
(resp., \eqref{eq:liconj-2}).
\end{proof}

\Cref{conj:li1998}, formally posed in \cite{li:1998}, was inspired by various eigenvalue perturbation results and stems from the reformulations (see also \cite{li:1993a}) of these results. In these reformulations, $\rank(G)\le 2$, so the inequalities as conjectured, if proven, are much more general than what may straightforwardly be implied by the corresponding perturbation results.
With \Cref{thm:main}, the conjectures for the case of the Frobenius norm have all been proven true with the corresponding constant $c_i=1$, implying the two quantities in each of the three problems in \cref{sec:intro} are actually equal.

Those results for the case of the Frobenius norm are somewhat foretold by the research in
Tilli~\cite{till:2001} and Romeo and Tilli~\cite{roti:2001}, where it is shown that
each of the inequalities in \Cref{conj:li1998}, if true in its generality for any matrix $G$, implies
the involved unitarily invariant norm must be uniformly equivalent to the Frobenius norm.
\Cref{cor:roti2001',cor:main} say that the converses are also true.

\begin{table}[t]
\renewcommand{\arraystretch}{1.4}
\caption{\small Validity of the conjectures for $G=\ba\bone^{\T}-\bone\bb^{\T}$.
    % with $\ba\equiv [a_i],\,\bb\equiv [b_i]\in\bbC^n$
               }\label{tbl:summary}
\centerline{\small
\begin{tabular}{|c|c|c|c|c|}
  \hline
    & $\ba\equiv [a_i],\,\bb\equiv [b_i]$ & $\|\cdot\|_{\UI}$ & admissible constant  & References \\ \hline\hline
\multirow{3}{*}{\eqref{eq:liconj-1} holds}
    & $\ba,\,\bb\in\bbR^n$ & $\|\cdot\|_{\UI}$ &  $c_1=1$    &  \cite[Thm.~1]{bhdk:1991} \\ \cline{2-5}
    & all $|a_i|=|b_i|=1$  & $\|\cdot\|_{\UI}$ & $c_1=\pi/2$ & \cite[Cor.~5]{bhdk:1991}, \cite[Thm.~5.2]{bhdm:1983}  \\ \cline{2-5}
    & all $|a_i|=|b_i|=1$  & $\|\cdot\|_2$     &  $c_1=1$    & \cite{bhda:1984}, \cite[Cor.~5]{bhdk:1991} \\ \hline
\multirow{3}{*}{\eqref{eq:liconj-2} holds}
    & $\ba,\,\bb\in\bbR^n$ & $\|\cdot\|_{\UI}$ & $c_2=1$         &  \cite[Thm.~2.1]{bhkl:1997a} \\ \cline{2-5}
    & all $|a_i|=|b_i|=1$  & $\|\cdot\|_{\UI}$ & $c_2=(\pi/2)^2$ &  \cite[Thm.~2.2]{bhkl:1997a} \\ \cline{2-5}
    & all $|a_i|=|b_i|=1$  & $\|\cdot\|_2$     & $c_2=1$         &  \cite{bhda:1984}, \cite[Thm.~2.2]{bhkl:1997a} \\ \hline
\multirow{4}{*}{\eqref{eq:liconj-0} holds}
    & $\ba,\,\bb\in\bbR^n$ & $\|\cdot\|_{\UI}$ & $c_0=1$         &  \cite{bhat:1997,stsu:1990} \\ \cline{2-5}
    & all $|a_i|=|b_i|=1$  & $\|\cdot\|_{\UI}$ & $c_0=\pi/2$     &  \cite[Thm.~5.2]{bhdm:1983} \\ \cline{2-5}
    & all $|a_i|=|b_i|=1$  & $\|\cdot\|_2$     & $c_0=1$         &  \cite{bhda:1984} \\ \cline{2-5}
    & $\ba,\,\bb\in\bbC^n$ & $\|\cdot\|_2$     & $c_0<\frac {\pi}2\int_0^{\pi}\frac {\sin t}t\rmd\! t$
                &  \cite[Thm.~5.1]{bhdm:1983}, \cite{bhdk:1989} \\ \hline
\end{tabular}
}
\end{table}

The results of Romeo and Tilli do not exclude the validity of \Cref{conj:li1998} for
unitarily invariant norms other than the ones that are uniformly equivalent to the Frobenius norm
if additional constraints are imposed on $G$. In fact,
for rank-one $G$, Romeo and Tilli~\cite[Theorem~1.2]{roti:2001}
proved \eqref{eq:liconj-1} with $c_1=1$ and the unitary restriction
\eqref{eq:liconj-2'} with $c_2=1$.
They explicitly left the corresponding nonsingular version of
Problem~3 unresolved; see \cite[Remark~3.1]{roti:2001}.
%Romeo and Tilli~\cite{roti:2001} already showed that both conjectures in \Cref{conj:li1998}  hold with $c_1=c_2=1$ if $\rank(G)=1$.
A natural next focus would be on the case $\rank(G)=2$ as in those coming from the
eigenvalue perturbation theory \cite{li:1998}. There are also a number of special cases in $G$,
from eigenvalue perturbations,
for which some forms of the conjectures hold.
\Cref{tbl:summary} provides a summary
%, lists corollaries of known eigenvalue perturbation results
%towards conjectures \eqref{eq:liconj-1}, \eqref{eq:liconj-2}, and \eqref{eq:liconj-0}
stemming from the additive perturbation theory \cite{bhat:2007,stsu:1990,li:2014HLA}
for $G=\ba\bone^{\T}-\bone\bb^{\T}$, where $\ba\equiv [a_i],\,\bb\equiv [b_j]\in\bbC^n$, and $\bone\in\bbR^n$ is the vector of all ones. Additionally, some of the results in \cite{bhli:1996,li:1993a,li:1994a,li:2003} can be reformulated
similarly with $G=\ba\bd^{\T}-\bc\bb^{\T}$ where $\ba,\bb,\bc,\bd\in\bbR^n$ such that
$\ba^{\circ 2}+\bc^{\circ 2}=\bb^{\circ 2}+\bd^{\circ 2}=\bone$. Details are omitted.

In connection with the relative perturbation theory \cite{li:2014HLA}, we have one implied result
regarding \eqref{eq:liconj-0}: for $G\equiv [g_{ij}]=\left[\frac {a_i^2-b_j^2}{a_ib_j}\right]\in\bbR^{n\times n}$ with
$a_i\ge 0,\,b_j\ge 0$ for $1\le i,j\le n$, where it is understood that $g_{ij}=0$ both $a_i=b_j=0$ and $\infty$ if one
of $a_i$ and $b_j$ is $0$ but the other is positive. 
Conjecture \eqref{eq:liconj-0} holds for any unitarily invariant norm and with $c_0=1$. This is a corollary of
\cite[ineq.~(3.4)]{li:1994b96} for the spectral norm but of \cite[ineq.~(5.8)]{lima:1999}
for a general unitarily invariant norm.

The structured matrices above have rank at most two.
This raises the question of whether \eqref{eq:liconj-1},
\eqref{eq:liconj-2}, and \eqref{eq:liconj-0} hold under the sole
assumption $\rank(G)=2$, without these additional structural constraints.
For $G=\ba\bone^{\T}-\bone\bb^{\T}$ with
$\ba,\bb\in\bbC^n$ and $\|\cdot\|_{\UI}=\|\cdot\|_2$, 
finding
the smallest constant $c_0$ in \eqref{eq:liconj-0} that is valid
uniformly for all $n\ge1$ and all $\ba,\bb\in\bbC^n$ is a long standing open problem
and is recently collected as
Problem SP-07 of~\cite{townsend2026openproblemsnla}.
\iffalse
More precisely, for normal matrices $A,B\in\bbC^{n\times n}$ with
eigenvalues $a_1,\ldots,a_n$ and $b_1,\ldots,b_n$, respectively,
counted with algebraic multiplicity, define
$$
d_\infty(A,B):=\min_{\tau}\max_{1\le i\le n}|a_i-b_{\tau(i)}|,
$$
where $\tau$ ranges over all permutations of $\bbN_n$.
The problem is to determine the smallest constant $C$ such that
$d_\infty(A,B)\le C\|A-B\|_2$ for every such pair in every dimension.
To see the equivalence, set $D_a=\diag(a_1,\ldots,a_n)$ and
$D_b=\diag(b_1,\ldots,b_n)$. For every unitary $W$,
$$
\begin{aligned}
\|W\circ G\|_2
&=\|D_aW-WD_b\|_2
=\|D_a-WD_bW^{\HH}\|_2,\\
\min_{P\in\bbP_n}\|P\circ G\|_2
&=\min_{\tau}\max_{1\le i\le n}|a_i-b_{\tau(i)}|.
\end{aligned}
$$
After a simultaneous unitary change of basis, every such pair has
the form $(D_a,WD_bW^{\HH})$ for some unitary $W$.
Thus the same constants are admissible in the spectral-matching
inequality and in \eqref{eq:liconj-0} for this class of $G$.
\fi

So far, the Frobenius-norm consequences are omitted from the table because
the three conjectures have been proved true  with $c_0=c_1=c_2=1$ for the Frobenius norm and for any matrix $G$.
 
%%%%%%%%%%%
\section{Conclusions}\label{sec:concl}
In this paper, we resolve one of the conjectures in Li~\cite{li:1998}, i.e., \eqref{eq:liconj-2},  for the case of the Frobenius norm. Namely, we show
$$
\min_{W\ \mathrm{nonsingular}}
\|W\circ G\|_{\F}\,\|W^{-1}\circ G^{\T}\|_{\F}
      =\min_{P\ \mathrm{permutation}}\|P\circ G\|_{\F}^2.
$$
This, combined with Romeo-Tilli theorem, \Cref{thm:roti2001}, solves the conjecture in its generality for any $G$:
conjecture \eqref{eq:liconj-2} holds if and only if the involved unitarily invariant norm is uniformly equivalent to
the Frobenius norm with respect to matrix size. As a result, further questions concerning general unitarily invariant norm 
may rest in $G$ with certain structural constraints such as being of low-rank. In fact, restricted to unitary matrices $W$, Romeo and Tilli \cite[Theorem 1.2]{roti:2001} showed that the conjecture \eqref{eq:liconj-1} holds with $c_1=1$ if
$\rank(G)=1$ and also that if $\rank(G)=1$ then
$$
\min_{W\ \mathrm{unitary}}
\|W\circ G\|_{\UI}\,\|W^{-1}\circ G^{\T}\|_{\UI}
      =\min_{P\ \mathrm{permutation}}\|P\circ G\|_{\UI}^2.
$$
As the case $\rank(G)=2$ arises naturally  in the study of
matrix eigenvalue perturbation theory, the next logical question is: what if $\rank(G)\le 2$?

\section*{AI statement}
The authors used assistance from AI tools to develop ideas presented in this
paper.
The authors assume responsibility for all content.

\clearpage
\bibliographystyle{plain}
\bibliography{refs}

\end{document}